\documentclass[12pt,reqno]{amsart}

\usepackage{amsaddr}

\usepackage{xcolor}

\usepackage{a4wide}
\usepackage[T1]{fontenc}
\usepackage{amsmath}
\usepackage[swedish,english]{babel}
\usepackage{amssymb,bbm}
\usepackage{amsthm}
\usepackage{enumerate}

\usepackage{hyperref}

\usepackage{marginnote}

\newcommand{\Z}{\mathbb{Z}}

\newcommand{\R}{\mathbb{R}}
\newcommand{\C}{\mathbb{C}}

\newcommand{\Rskew}{R_e \rtimes_\sigma^\alpha G}

\newcommand{\Acal}{\mathcal{A}}

\DeclareMathOperator{\identity}{id}
\DeclareMathOperator{\Aut}{Aut}

\DeclareMathOperator{\Supp}{Supp}

\DeclareMathOperator{\Pic}{Pic}

\DeclareMathOperator{\im}{im}

\renewcommand{\emptyset}{\varnothing}

\theoremstyle{plain}
\newtheorem{theorem}{Theorem}[section]
\newtheorem{lemma}[theorem]{Lemma}
\newtheorem{prop}[theorem]{Proposition}
\newtheorem{corollary}[theorem]{Corollary}
\theoremstyle{definition}
\newtheorem{definition}[theorem]{Definition}

\newtheorem{exmp}[theorem]{Example}

\newtheorem{question}{Question}

\newtheorem{remark}[theorem]{Remark}

\begin{document}

\title{Bimodules in group graded rings}

\author{Johan \"{O}inert}
\address{Blekinge Institute of Technology,\\
Department of Mathematics and Natural Sciences,\\
SE-37179 Karlskrona, Sweden}
\email{johan.oinert@bth.se}

\date{2017-01-09}

\begin{abstract}
In this article we introduce the notion of a \emph{controlled} group graded ring.
Let $G$ be a group, with identity element $e$,
and let $R=\oplus_{g\in G} R_g$ be a unital $G$-graded ring.
We say that $R$
is
\emph{$G$-controlled}
if there is a one-to-one correspondence between
subsets of the group $G$
and (mutually non-isomorphic) $R_e$-bimodules in $R$,
given by $G \supseteq H \mapsto \oplus_{h\in H} R_h$.
For strongly $G$-graded rings, the property of being $G$-controlled is stronger than that of being simple.
We provide necessary and sufficient conditions for a general $G$-graded ring to be $G$-controlled.
We also give a
characterization of strongly $G$-graded rings
which are $G$-controlled.
As an application of our main results we give a description of all intermediate subrings $T$ with
$R_e \subseteq T \subseteq R$ of a $G$-controlled strongly $G$-graded ring $R$.
Our results generalize results for artinian skew group rings which were shown by Azumaya 70 years ago.
In the special case of skew group rings we obtain
an algebraic analogue of a recent result by Cameron and Smith on bimodules in crossed products of von Neumann algebras.
\end{abstract}

\maketitle

\pagestyle{headings}


\section{Introduction}

Recently, Cameron and Smith \cite{CameronSmith2015} studied bimodules
over a von Neumann algebra $M$ in the context of an inclusion
$M \subseteq M \rtimes_\alpha G$,
where $G$ is a group acting on $M$ by $*$-automorphisms and $M \rtimes_\alpha G$ is the corresponding crossed product von Neumann algebra.
They have shown \cite[Theorem 4.4(i)]{CameronSmith2015} that
if $G$ is a discrete group acting by outer $*$-automorphisms
on a simple\footnote{In fact their result is more general. They only assume that $M$ is a factor.} von Neumann algebra $M$, then there is a bijective correspondence between subsets of $G$ and $B$-closed (i.e. closed in the \emph{Bures-topology}) $M$-bimodules of $M \rtimes_\alpha G$.

It is natural to ask whether the same correspondence holds for a skew group ring,
which is the algebraic analogue of a crossed product von Neumann algebra.
To be more precise, if $G$ is a group which is acting
by outer automorphisms on a simple and unital ring $A$,
then we ask whether each $A$-bimodule, which is contained in the skew group ring $A \rtimes_\alpha G$,
is of the form $\oplus_{h\in H} A u_h$ for some subset $H \subseteq G$, where $\emptyset$ corresponds to the zero-module.
As it turns out, although it was not the main
focus of his investigation, in one of his proofs Azumaya has observed this correspondence
in the case when $G$ is finite \cite{Azumaya1946}.

The purpose of this article is to, in a systematic way, study the same type of correspondence in the more
general context of group graded rings.

Let $R$ be an associative and unital ring
and let $G$ be a multiplicatively written group with identity element $e\in G$. 
For subsets $X$ and $Y$ of $R$,
we let $XY$ denote the set of all finite sums of elements of the form $xy$, for $x\in X$ and $y\in Y$.
If there is a family $\{R_g\}_{g\in G}$ of additive subgroups
of $R$ such that
\begin{displaymath}
	R = \oplus_{g\in G} R_g \quad \text{and} \quad  R_g R_h \subseteq R_{gh}
\end{displaymath}
for all $g,h \in G$, then the ring $R$ is said to be \emph{$G$-graded} (or \emph{graded by $G$}).
A $G$-graded ring $R$ for which $R_g R_h = R_{gh}$ holds, for all $g,h\in G$, is said to be \emph{strongly $G$-graded}.

If $R$ is a $G$-graded ring, then one immediately observes that
$R_e$ is a subring of $R$ and that $1_R \in R_e$ (see e.g. \cite[Proposition 1.4]{Dade1980}).
For any $g\in G$, $R_g$ is an $R_e$-bimodule.
If $R$ is strongly $G$-graded, then for each $g\in G$, $R_g$ is finitely generated and projective as a left (right) $R_e$-module
(see \cite[Proposition 1.10]{Karpilovsky1987}).

Each subset of $G$ gives rise to an $R_e$-bimodule in $R$.
Indeed,
if $H$ is a subset of $G$, then $R_H = \oplus_{h\in H} R_h$ is an $R_e$-bimodule which is contained in $R$.
We let the empty set give rise to the zero-module, i.e. $R_\emptyset = \{0\}$.
It is natural to ask the following question:
\begin{quote}
\emph{When does every $R_e$-bimodule of $R$ arise in this way?}
\end{quote}

Let $\mathtt{Mod}_R(R_e)$ denote the set of 
$R_e$-bimodules which are contained in $R$, with scalar multiplication
coming from
the ring multiplication in $R$.
We make the following definition.

\begin{definition}[$G$-controlled ring]\label{DefBM}
A ring $R$ is said to be \emph{$G$-controlled}
if it is equipped with a $G$-gradation
such that the following two assertions hold:
\begin{enumerate}
	\item The map
\begin{displaymath}
	\varphi : \mathcal{P}(G) \to \mathtt{Mod}_R(R_e),
	\quad H \longmapsto \, \oplus_{h\in H} R_h
\end{displaymath}
is a bijection.
\item For $S,T\in \mathcal{P}(G)$, $\varphi(S) \cong \varphi(T)$ (as $R_e$-bimodules) if and only if $S=T$.
\end{enumerate}
\end{definition}

This article is organized as follows.

In Section \ref{Sec:PrelNot} we recall important notions which will be used in subsequent sections.
In Section \ref{Sec:CharGeneral} we give a complete characterization of $G$-controlled rings (see Theorem~\ref{MainResult}).
We also provide an example of a $G$-controlled ring which is not strongly $G$-graded (see Example~\ref{GcontrolledNotStrong}).
In Section \ref{Sec:CharStrongly} we first point out that $G$-controlled rings are often strongly $G$-graded (see Proposition~\ref{GradedSimpleStronglyEquivalence} and Remark~\ref{Rem:GcontrolledStrongly}).
We then 
give a characterization of strongly $G$-graded rings which are $G$-controlled (see Theorem~\ref{MainResult2}). We also specialize this result to $G$-crossed products (incl. skew group rings) to see how $G$-controlness is connected to outerness (see Corollary~\ref{MainResult:CrossedProducts} and Remark~\ref{Rem:OuternessAzumaya}). This shows how our result generalizes those of Azumaya \cite{Azumaya1946}.
In Section \ref{Sec:Subrings} we give a description of all intermediate subrings $T$ where $R_e \subseteq T \subseteq R$
of a strongly $G$-graded and $G$-controlled ring $R$ (see Proposition~\ref{Prop:Subring}).
In Section \ref{Sec:Disc} we present some simplicity results on strongly $G$-graded rings and explain how they are related to our investigation of $G$-controlled rings. We also present some open questions (see Questions~\ref{Q1}, \ref{Q2} and \ref{Q3}).

\section{Preliminaries and notation}\label{Sec:PrelNot}

The centralizer of
a non-empty subset $S$ of a ring $T$ will be denoted by $C_T(S)$ and is defined as the set
of all elements of $T$ that commute with each element of $S$. The centre
of $T$ is defined as $C_T(T)$ and will be denoted by $Z(T)$.
The group of multiplication invertible elements of $T$ will be denoted by $U(T)$.

Let $R=\oplus_{g\in G} R_g$ be a $G$-graded ring.
Each element $x \in R$ may be written as $x=\sum_{g\in G} x_g$ where $x_g \in R_g$ is unique for each $g\in G$,
and zero for all but finitely many $g\in G$.
For $g\in G$ we define a map
\begin{displaymath}
	E_g : R \to R_g, \quad x=\sum_{h\in G}x_h \mapsto x_g.
\end{displaymath}
Notice that $E_g$ is an $R_e$-bimodule homomorphism.
The \emph{support} of $r\in R$ is defined as $\Supp(r)=\{g\in G \mid E_g(r)\neq 0 \}$.
An ideal $I$ of a $G$-graded ring $R$ is said to be \emph{graded} if $I=\oplus_{g\in G} (I \cap R_g)$ holds.
The ring $R$ is said to be \emph{graded simple} if $R$ and $\{0\}$ are the only two graded ideals of $R$.

Recall that
$R=\oplus_{g\in G} R_g$ is said to be a \emph{$G$-crossed product}
if $E_g(R) \cap U(R) \neq \emptyset$, for each $g\in G$.
In that case, we may choose an invertible $u_g \in R_g$, for each $g\in G$.
Pick $u_e=1$.
It is clear that $R_g = R_e u_g = u_g R_e$ and that the set $\{u_g \mid g\in G\}$ is a basis for $R$ as a left
(and right) $R_e$-module. We now define two maps:
\begin{displaymath}
	\sigma : G \to \Aut(R_e) \text{ by } \sigma_g(a)=u_g a u_{g^{-1}} \text{ for } g\in G, a\in R_e,
\end{displaymath}
and
\begin{displaymath}
	\alpha : G \times G \to U(R_e) \text{ by } \alpha(g,h)=u_g u_h u_{gh}^{-1} \text{ for } g,h \in G.
\end{displaymath}
One may now show that the following holds for any $g,h,t\in G$ and $a\in R_e$ (see e.g. \cite[Proposition 1.4.2]{NVO2004}):
\begin{enumerate}
	\item $\sigma_g(\sigma_h(a))=\alpha(g,h) \sigma_{gh}(a) \alpha(g,h)^{-1}$
	\item $\alpha(g,h)\alpha(gh,t)=\sigma_g(\alpha(h,t)) \alpha(g,ht)$
	\item $\alpha(g,e)=\alpha(e,g)=1_R$.
\end{enumerate}
Any two homogeneous elements $a\in R_g$ and $b\in R_h$
may be expressed as $a=a_1 u_g$ and $b=b_1 u_h$, with $a_1,b_1\in R_e$ and their product is
\begin{displaymath}
ab = (a_1 u_g)(b_1 u_h) = a_1 ( u_g b_1 u_g^{-1}) u_g u_h
= a_1 ( u_g b_1 u_g^{-1}) (u_g u_h u_{gh}^{-1}) u_{gh}
= a_1 \sigma_g(b_1) \alpha(g,h) u_{gh}.
\end{displaymath}

Important examples of $G$-crossed products are given by e.g. skew group rings, twisted group rings
and group rings.
It is not difficult to see that $G$-crossed products are necessarily strongly $G$-graded.
However, as
e.g. Example \ref{Example:Matrix} demonstrates, not all strongly $G$-graded rings are $G$-crossed products.

\section{A characterization of $G$-controlled rings}\label{Sec:CharGeneral}

In this section we give a characterization of $G$-controlled rings (see Theorem~\ref{MainResult}).
We begin by finding necessary conditions for a $G$-graded ring to be $G$-controlled.

\begin{prop}\label{Prop:GeneralNecessaryConditions}
Let $G$ be a group and let $R$ be a $G$-graded ring.
If $R$ is $G$-controlled, then the following five assertions hold:
\begin{enumerate}[{\rm (i)}]
	\item\label{PropGenNec:RgCongRh} $R_g \cong R_h$ (as $R_e$-bimodules) if and only if $g=h$;
	\item\label{PropGenNec:SimpleModule} $R_g$ is a (non-zero) simple $R_e$-bimodule for each $g\in G$;
	\item\label{PropGenNec:SimpleRe} $R_e$ is a simple ring;
	\item\label{PropGenNec:Centralizer} $C_R(R_e)=Z(R_e)$;
	\item\label{PropGenNec:GradedIdeal} Every ideal of $R$ is graded.
\end{enumerate}
\end{prop}

\begin{proof}
Let $R$ be a $G$-controlled ring, and let $\varphi$ be defined as in Definition \ref{DefBM}.

\noindent \eqref{PropGenNec:RgCongRh}:
Take $g,h\in G$. Notice that $\varphi(\{g\}) \cong \varphi(\{h\})$ if and only if $g=h$, i.e.
$R_g \cong R_h$ if and only if $g=h$.

\noindent \eqref{PropGenNec:SimpleModule}: Take $g\in G$.
By the injectivity of $\varphi$ we get that $R_g \neq \{0\}$,
and by the surjectivity of $\varphi$,
$R_g$ can not contain any proper non-zero $R_e$-submodule.
Thus, $R_g$ is a simple $R_e$-bimodule.

\noindent \eqref{PropGenNec:SimpleRe}: This follows immediately from \eqref{PropGenNec:SimpleModule}.

\noindent \eqref{PropGenNec:Centralizer}: Take $g\in G$ and let $x_g \in C_R(R_e) \cap R_g$ be non-zero.
Define $f : R_e \to R_g, r \mapsto r x_g$. Clearly, $f$ is an $R_e$-bimodule homomorphism.
Using \eqref{PropGenNec:SimpleModule} we conclude that $\ker f = \{0\}$ and that $\im f = R_g$, i.e.
$f$ is an isomorphism. From \eqref{PropGenNec:RgCongRh} we get $g=e$. Hence, $C_R(R_e)=Z(R_e)$.

\noindent \eqref{PropGenNec:GradedIdeal}: Every ideal $I$ of $R$ is an $R_e$-bimodule. Hence $I=\oplus_{h\in H} R_h$ for some
subset $H\subseteq G$. In particular, $I$ is graded.
\end{proof}

We now begin our search for sufficient conditions for $G$-controlness
by showing the following essential lemma.

\begin{lemma}\label{MainLemma}
Let $S$ be a unital ring and let $M$ and $N$ be (non-zero) simple $S$-bimodules
which are non-isomorphic.
For any $x\in M \setminus \{0\}$ and $y \in N \setminus \{0\}$ there is some $n\in \Z_+$
and $s_1^{(1)},\ldots,s_n^{(1)},s_1^{(2)},\ldots,s_n^{(2)} \in S$ such that
$\sum_{i=1}^n s_i^{(1)} x s_i^{(2)} \neq 0$ and $\sum_{i=1}^n s_i^{(1)} y s_i^{(2)} = 0$.
\end{lemma}

\begin{proof}
Take $x\in M \setminus \{0\}$ and $y \in N \setminus \{0\}$. We notice that $SxS=M$ and $SyS=N$.
Seeking a contradiction, suppose that $\sum_{i=1}^n s_i^{(1)} x s_i^{(2)} =0$ whenever $\sum_{i=1}^n s_i^{(1)} y s_i^{(2)} = 0$.
We define a map $f : N \to M, \quad \sum_{i=1}^n s_i^{(1)} y s_i^{(2)} \mapsto \sum_{i=1}^n s_i^{(1)} x s_i^{(2)}$.
By our assumption $f$ is a well-defined homomorphism of $S$-bimodules.
Moreover, by the unitality of $S$ and the simplicity of $M$ and $N$ we conclude that $f$ is an isomorphism. This is a contradiction.
\end{proof}

\begin{prop}\label{Prop:Surjectivity}
Let $G$ be a group and let $R=\oplus_{g\in G} R_g$ be a $G$-graded ring.
Suppose that $R_g$ is a (non-zero) simple $R_e$-bimodule, for each $g\in G$,
and that $R_g \cong R_h$ if and only if $g=h$, for $g,h\in G$.
If $P$ is an $R_e$-bimodule which is contained in $R$ and $x\in P \setminus \{0\}$,
then $R_g$ is an $R_e$-submodule of $P$, for every $g\in \Supp(x)$.
In particular, $P=\oplus_{s\in S} R_s$ for some subset $S\subseteq G$.
\end{prop}

\begin{proof}
Take $x\in P \setminus \{0\}$ and $g\in \Supp(x)$.
Choose $y\in P \setminus \{0\}$ such that $|\Supp(y)|$ is minimal amongst all elements satisfying $g\in \Supp(y) \subseteq \Supp(x)$.
Seeking a contradiction, suppose that $|\Supp(y)|>1$.
Choose some $h \in \Supp(y) \setminus \{g\}$.
Using Lemma \ref{MainLemma}, with $S=R_e$, $M=R_g$ and $N=R_h$, we conclude that there is some
$y' \in P$ such that $|\Supp(y')| < |\Supp(y)|$ and $g \in \Supp(y') \subseteq \Supp(x)$.
This is a contradiction.
Hence, $P \cap R_g \neq \{0\}$. Using that $R_g$ is a simple $R_e$-bimodule we conclude that $R_g$ is an $R_e$-submodule of $P$.
From this it follows that $P=\oplus_{s\in S} R_s$ for some subset $S \subseteq G$.
\end{proof}

\begin{lemma}\label{Lemma:UpToIso}
Let $G$ be a group and let $R$ be a
$G$-graded ring
such that $R_g$ is a simple $R_e$-bimodule for each $g\in G$.
The following two assertions are equivalent:
\begin{enumerate}[{\rm (i)}]
	\item\label{UpToIso:RSRT} $R_S \cong R_T$ if and only if $S=T$ (where $S,T\subseteq G$).
	\item\label{UpToIso:RgRh} $R_g \cong R_h$ if and only if $g=h$ (where $g,h\in G$);
\end{enumerate}
\end{lemma}

\begin{proof}
\eqref{UpToIso:RSRT}$\Rightarrow$\eqref{UpToIso:RgRh}: This is trivial.\\
\eqref{UpToIso:RgRh}$\Rightarrow$\eqref{UpToIso:RSRT}:
The ''if'' statement is trivial.
Now we show the ''only if'' statement.
Suppose that $f : R_S \to R_T$ is an $R_e$-bimodule isomorphism.
Take $s\in S$.
Then $f(R_s)$ is a simple $R_e$-submodule of $R_T$.
By Proposition \ref{Prop:Surjectivity} we conclude that $f(R_s)=R_t$ for some $t\in T$.
This shows that $R_s \cong R_t$ and by \eqref{UpToIso:RgRh} we get $s=t$.
Thus, $s \in T$.
Using that $s$ was chosen arbitrarily, we get $S \subseteq T$.
In the same way we can show that $T\subseteq S$. This shows that $S=T$.
\end{proof}

We are now ready to prove the first main result of this article.

\begin{theorem}\label{MainResult}
Let $G$ be a group and let $R$ be a $G$-graded ring.
Then $R$ is $G$-controlled if and only if
(a) $R_g$ is a simple $R_e$-bimodule, for each $g\in G$;
and
(b) $R_g \cong R_h$ if and only if $g=h$, for $g,h\in G$.
\end{theorem}

\begin{proof}
The ''only if'' statement follows from Proposition~\ref{Prop:GeneralNecessaryConditions}.

We now show the ''if'' statement.
Suppose that (a) and (b) hold.
By using the maps $E_g$, for $g\in G$, we may conclude that $\varphi$ (in Definition~\ref{DefBM}) is injective.
By Proposition~\ref{Prop:Surjectivity}, $\varphi$ is surjective.
Hence, $\varphi$ is a bijection.
By Lemma \ref{Lemma:UpToIso} we get that assertion (2) of Definition~\ref{DefBM} holds.
This shows that $R$ is $G$-controlled.
\end{proof}

We shall now present an example of a $G$-graded ring which is $G$-controlled but not strongly $G$-graded.
Notice that this ring is not simple (cf. Proposition~\ref{GradedSimpleStronglyEquivalence}).

\begin{exmp}\label{GcontrolledNotStrong}
Consider the first Weyl algebra $\Acal_1=\C\langle x, y \rangle/(yx-xy-1)$. Recall that $\Acal_1$ is a simple Noetherian domain.
Take any automorphism $\alpha : \Acal_1 \to \Acal_1$ satisfying $\{a \in \Acal_1 \mid \alpha(a)=a\}=\C$ and $\alpha^n \neq \identity_{\Acal_1}$ if $n\neq 0$. (We may e.g. choose $\alpha$ defined by $x\mapsto x-1$ and $y\mapsto y+1$.)
Let us now define a free left $\Acal_1$-module $R=\oplus_{n \in \Z} \Acal_1 u_n$ with basis $\{u_n\}_{n\in \Z}$.
We define a multiplication on $R$ by
\begin{displaymath}
	a u_n \cdot b u_m =
	\left\{
	\begin{array}{ll}
			a \alpha^n(b) u_{n+m} & \text{if } n=0 \text{ or } m =0 \\
			0 & \text{otherwise.}
	\end{array}
	\right.
\end{displaymath}
It is not difficult to verify that this turns $R$ into a unital and associative ring which is $\Z$-graded, but not strongly $\Z$-graded.
Moreover, $R_n= \Acal_1 u_n$ is a simple $R_0$-bimodule for each $n\in \Z$.
Take $n,m\in \Z$ with $n\neq m$. We claim that $R_n$ and $R_m$ can not be isomorphic as $R_0$-bimodules.
Seeking a contradiction, suppose that there is an $R_0$-bimodule isomorphism
$f : R_n=\Acal_1 u_n \to R_m=\Acal_1 u_m$.
Then there is some $c\in \Acal_1$ such that $f(1 u_n)=c u_m$.
But then $c^2 u_m = c(c u_m) = c f(u_n) = f(c u_n) = f(u_n \alpha^{-n}(c)) = 
f(u_n) \alpha^{-n}(c) = c u_m \alpha^{-n}(c) = c \alpha^{m-n}(c) u_m$.
Using that $R_m$ is a free left $\Acal_1$-module, we get $c^2=c \alpha^{m-n}(c)$.
By our assumptions we conclude that $c\in \C = Z(\Acal_1)$.
Now, take any $b\in \Acal_1 \setminus \C$.
Then, we get
\begin{displaymath}
c \alpha^m(b) u_m = c u_m b = f(u_n)b = f(u_n b) = f( \alpha^{n}(b) u_n )
= \alpha^{n}(b) f(u_n) = \alpha^{n}(b) c u_m
\end{displaymath}
and hence
$c \alpha^m(b) = \alpha^{n}(b) c = c \alpha^{n}(b)$.
By our assumptions this is a contradiction.
Using Theorem~\ref{MainResult} we conclude that $R$ is a $\Z$-controlled ring.
\end{exmp}

\section{A characterization of $G$-controlled strongly $G$-graded rings}\label{Sec:CharStrongly}

In this section we give a characterization of $G$-controlled rings which are strongly $G$-graded (see Theorem~\ref{MainResult2}).
We begin by noticing that by Example~\ref{GcontrolledNotStrong} there exist $G$-controlled which are not strongly $G$-graded. In many cases, however, $G$-controlness will force the gradation to be strong (see Proposition~\ref{GradedSimpleStronglyEquivalence}).

\begin{lemma}\label{Prop:GradedSimpleStronglyGraded}
Let $G$ be a group and let $R$ be a $G$-graded ring.
If $R$ is $G$-controlled, then the following two assertions hold:
\begin{enumerate}[{\rm (i)}]
	\item\label{Lem:GSSGi} If $R_g R_{g^{-1}} = \{0\}$, then $R_{g^{-1}} R_g = \{0\}$;
	\item\label{Lem:GSSGii} If $R_g R_{g^{-1}} = R_e$, then $R_{g^{-1}} R_g = R_e$.
\end{enumerate}
\end{lemma}

\begin{proof}
We first notice that by Proposition~\ref{Prop:GeneralNecessaryConditions}\eqref{PropGenNec:SimpleModule},
$R_g$ is a (non-zero) simple $R_e$-bimodule for each $g\in G$.

\noindent \eqref{Lem:GSSGi}:
Suppose that $\{0\} = R_g R_{g^{-1}}$ holds.
Seeking a contradiction, suppose that $R_{g^{-1}} R_g \neq \{0\}$.
Then $R_{g^{-1}} R_g = R_e$ and hence
$R_g = R_g R_e = R_g (R_{g^{-1}} R_g) = \{0\} R_g = \{0\}$.
By Proposition~\ref{Prop:GeneralNecessaryConditions}\eqref{PropGenNec:SimpleModule}, this is a contradiction.

\noindent \eqref{Lem:GSSGii}:
This follows from \eqref{Lem:GSSGi}.
\end{proof}

\begin{prop}\label{GradedSimpleStronglyEquivalence}
Let $G$ be a group and let $R$ be a $G$-graded ring.
If $R$ is $G$-controlled, then the following four assertions are equivalent:
\begin{enumerate}[{\rm (i)}]
	\item\label{Prop:EQ1} $R$ is graded simple;
	\item\label{Prop:EQ2} $R$ is simple;
	\item\label{Prop:EQ3} $R$ is strongly $G$-graded;
	\item\label{Prop:EQ4} The $G$-gradation on $R$ is left (and right) non-degenerate.
\end{enumerate}
\end{prop}

\begin{proof}
\noindent \eqref{Prop:EQ3}$\Rightarrow$\eqref{Prop:EQ4}: This is clear.

\noindent \eqref{Prop:EQ4}$\Rightarrow$\eqref{Prop:EQ2}:
Suppose that the $G$-gradation on $R$ is left (and right) non-degenerate. 
Let $I$ be a non-zero ideal of $R$. It follows from Proposition~\ref{Prop:GeneralNecessaryConditions}\eqref{PropGenNec:GradedIdeal}
that $I$ is graded. Hence, by the assumption and 
Proposition~\ref{Prop:GeneralNecessaryConditions}\eqref{PropGenNec:SimpleRe}
we conclude that $R_e \subseteq I$. Thus, $I=R$.

\noindent \eqref{Prop:EQ2}$\Rightarrow$\eqref{Prop:EQ1}: This is clear.

\noindent \eqref{Prop:EQ1}$\Rightarrow$\eqref{Prop:EQ3}:
Suppose that $R$ is graded simple.
Take $g\in G$ and notice that $R_g$ is non-zero.
By graded simplicity there are some $s,t\in G$ such that $sgt=e$ and $R_s R_g R_t = R_e$.
From this we get
\begin{equation}\label{Lem:GSSGfirst}
	R_s R_g R_t R_{t^{-1}} = R_e R_{t^{-1}} = R_{t^{-1}} \neq \{0\}.
\end{equation}
Hence, $R_t R_{t^{-1}} \neq \{0\}$ and therefore $R_t R_{t^{-1}} = R_e$.
From this we get
$R_s R_g = R_{t^{-1}}$ and $R_{t^{-1}} R_t = R_e$, using \eqref{Lem:GSSGii}.
From \eqref{Lem:GSSGfirst} we get
\begin{equation}\label{Lem:GSSGsecond}
	R_t R_s R_g = R_t R_{t^{-1}} = R_e
\end{equation}
and hence $tsg=e$, i.e. $ts=g^{-1}$.
Since $R_t R_s \subseteq R_{g^{-1}}$, this shows that $R_{g^{-1}} R_g \neq \{0\}$
which yields $R_{g^{-1}} R_g =R_e \ni 1_R$. Hence, $R$ is a strongly $G$-graded ring.
\end{proof}

\begin{remark}\label{Rem:GcontrolledStrongly}
If a $G$-controlled ring is e.g. crystalline graded \cite{NVO2008} or epsilon-strongly graded \cite{NOP2016}, then it is necessarily strongly $G$-graded. This follows from Proposition~\ref{GradedSimpleStronglyEquivalence}
and the fact that both crystalline graded rings and epsilon-strongly graded rings are left (and right) non-degenerate (cf. \cite[Definition 2]{OL2012}).
\end{remark}

Recall that if $T$ is a ring, then a $T$-bimodule $M$ is said to be \emph{invertible}
if there is a $T$-bimodule $N$ such that $M \otimes_T N \cong T$ and $N \otimes_T M \cong T$.
The Picard group of a ring $T$, denoted by $\Pic(T)$, consists of all equivalence classes of invertible $T$-bimodules
and the group operation is given by $\otimes_T$.
Using that $R$ is strongly $G$-graded, the map
$\psi : G \to \Pic(R_e), g \mapsto [R_g]$ is a group homomorphism (see e.g. \cite[Corollary 3.1.2]{NVO2004}).

For strongly $G$-graded rings, we record the following observation.

\begin{lemma}\label{lemma:CentralizerPicard}
Let $G$ be a group and let $R$ be a strongly $G$-graded ring.
Consider the following assertions:
\begin{enumerate}[{\rm (i)}]
	\item\label{lemma:CenPicI} $C_R(R_e)=Z(R_e)$;
	\item\label{lemma:CenPicII} The group homomorphism $\psi : G \to \Pic(R_e), g \mapsto [R_g]$ is injective.
\end{enumerate}
The following conclusions hold:
\begin{enumerate}[{\rm (a)}]
	\item\label{lemma:CenPicA} {\rm (i)} implies {\rm (ii)};
	\item\label{lemma:CenPicB} If $R_e$ is a simple ring,
	then {\rm (i)} holds if and only if {\rm (ii)} holds;
	\item\label{lemma:CenPicC} If $R$ is $G$-controlled, then both {\rm (i)} and {\rm (ii)} hold.
\end{enumerate}
\end{lemma}

\begin{proof}
\eqref{lemma:CenPicA}: Suppose that \eqref{lemma:CenPicI} holds.
Take $g\in G$ such that $R_g \cong R_e$, as $R_e$-bimodules.
Then there is an $R_e$-bimodule isomorphism
$f : R_e \to R_g$.
We notice that $0\neq f(1_R) \in R_g$.
For any $a\in R_e$ we have
$a f(1_R)=f(a 1_R)=f(1_R a) = f(1_R)a$,
showing that $f(1_R) \in C_R(R_e)=Z(R_e) \subseteq R_e$.
Thus, $g=e$.
This shows that $\psi$ is injective.

\eqref{lemma:CenPicB}:
Let $R_e$ be a simple ring and suppose that \eqref{lemma:CenPicII} holds.
Notice that $C_R(R_e)$ is a $G$-graded ring.
Take $g\in G$ and a non-zero $x_g \in C_R(R_e) \cap R_g$.
The set $I=x_g R_{g^{-1}} \subseteq R_e$ is a non-zero
ideal of $R_e$.
Indeed, by the strong gradation we get $x_g R_{g^{-1}} \neq \{0\}$
and from the fact that $R_{g^{-1}}$ is an $R_e$-bimodule
and that $x_g \in C_R(R_e)$, it follows that $I$ is an ideal of $R_e$.
By simplicity of $R_e$ we get $I=R_e$.
In particular, there is some $y_{g^{-1}}\in R_{g^{-1}}$
such that $x_g y_{g^{-1}} = 1_R$.
Symmetrically we get that $R_{g^{-1}} x_g = R_e$
which yields that $x_g$ also has a left inverse.
Hence, $x_g$ is invertible.

$R_e x_g \subseteq R_g$ and $R_g y_{g^{-1}} \subseteq R_e$.
Using that $y_{g^{-1}} x_g = 1_R$ we get $R_g \subseteq R_e x_g$.
This shows that $R_g = R_e x_g$.

Notice that $f : R_e \to R_g=R_e x_g, r \mapsto r x_g$
is an isomorphism of $R_e$-bimodules.

By injectivity of $\psi$ we conclude that $g=e$.
Hence, $C_R(R_e)\subseteq R_e$ which yields $C_R(R_e)=Z(R_e)$.

\eqref{lemma:CenPicC}:
This follows from \eqref{lemma:CenPicA}
and
Proposition~\ref{Prop:GeneralNecessaryConditions}\eqref{PropGenNec:Centralizer}.
\end{proof}

We are now ready to prove the second main result of this article.

\begin{theorem}\label{MainResult2}
Let $G$ be a group and let $R$ be a strongly $G$-graded ring.
The following three assertions are equivalent:
\begin{enumerate}[{\rm (i)}]
	\item $R$ is $G$-controlled;
	\item $R_g$ is a simple $R_e$-bimodule, for every $g\in G$, and $C_R(R_e)=Z(R_e)$;
	\item $R_g$ is a simple $R_e$-bimodule, for every $g\in G$, and the group homomorphism $\psi : G \to \Pic(R_e), g \mapsto [R_g]$ is injective.
\end{enumerate}
\end{theorem}

\begin{proof}
This follows from Lemma~\ref{lemma:CentralizerPicard} and Theorem~\ref{MainResult}.
\end{proof}

By combining
Theorem~\ref{MainResult2}
and Proposition~\ref{GradedSimpleStronglyEquivalence}
we get the following generalization of \cite[Theorem 4(1)]{Azumaya1946}.

\begin{corollary}\label{Cor:StrongSimplicity}
If $R_g$ is a simple $R_e$-bimodule for each $g\in G$, and $C_R(R_e)=Z(R_e)$,
then the strongly $G$-graded ring $R$ is a simple ring.
\end{corollary}

The following corollary is an algebraic analogue of 
\cite[Theorem 4.4(i)]{CameronSmith2015}.

\begin{corollary}\label{MainResult:CrossedProducts}
Let $G$ be a group and let $R$ be a $G$-crossed product.
The following three assertions are equivalent:
\begin{enumerate}[{\rm (i)}]
	\item $R$ is $G$-controlled;
	\item $R_e$ is a simple ring and $C_R(R_e)=Z(R_e)$;
	\item $R_e$ is a simple ring and for every invertible $u_g \in R_g$, $g\neq e$,
	the automorphism of $R_e$, defined by $\sigma_g(a)=u_g a u_g^{-1}$ for $a\in R_e$,
	is outer.
\end{enumerate}
\end{corollary}

\begin{proof}
Put $R=\Rskew$.\\
(i)$\Rightarrow$(ii):
This follows immediately from Proposition~\ref{Prop:GeneralNecessaryConditions}.

(ii)$\Rightarrow$(iii):
Suppose that (ii) holds.
Take $g\in G$. Suppose that $\sigma_g$ is inner, i.e. there is some invertible $v\in R_e$ such that
$\sigma_g(a)=u_g a u_{g^{-1}}=vav^{-1}$ holds for all $a\in R_e$.
From this we get that $av^{-1} u_g = v^{-1} u_g a$ holds for all $a\in R_e$.
Hence, $v^{-1} u_g \in C_R(R_e)=Z(R_e) \subseteq R_e$ and therefore we must have $g=e$.
This shows that (iii) holds.

(iii)$\Rightarrow$(i):
Suppose that (iii) holds.
We begin by noting that $C_{R}(R_e)$ is a $G$-graded subring of $R$.
Suppose that $a_g u_g \in C_{R}(R_e)$, for some $g\in G$.
By definition, $r a_g u_g = a_g u_g r$ for each $r\in R_e$.
Hence, $ra_g = a_g \sigma_g(r)$ for each $r\in R_e$.
From this we get that $a_g R_e = R_eA a_g$ is a non-zero two-sided ideal of $R_e$.
By simplicity of $R_e$ we conclude that $a_g$ is invertible.
Hence, $a_g^{-1} r a_g = \sigma_g(r)$ for each $r\in R_e$.
In other words, $\sigma_g$ is inner.
By the assumption on outerness we conclude that $g=e$.
This shows that $C_{R}(R_e) \subseteq R_e$, from which we get $C_R(R_e)=Z(R_e)$.
Using that $R_e$ is a simple ring we conclude that $R_e u_g$ is a simple $R_e$-bimodule for each $g\in G$.
The desired conclusion now follows directly from Theorem~\ref{MainResult2}.
\end{proof}

\begin{remark}\label{Rem:OuternessAzumaya}
For a skew group ring $A \rtimes_\sigma G$, 
Corollary~\ref{MainResult:CrossedProducts}(iii) means that $A$ is a simple ring and that the action of $G$ on $A$ is \emph{outer}
(see e.g \cite{Montgomery1980} or \cite{O2014}). 
\end{remark}

\begin{remark}
(a) Suppose that $R_e$ is a finite-dimensional central simple algebra.
By the Skolem--Noether theorem, every automorphism of $R_e$ is inner.
Hence, no $G$-controlled skew group ring $R$ (over $R_e$) can exist.

(b) Recall that each non-identity automorphism of the first Weyl algebra $\Acal_1$ is outer.
Hence, by taking any non-identity automorphism $\sigma_1 : \Acal_1 \to \Acal_1$ we may form
a $\Z$-controlled skew group ring $\Acal_1 \rtimes_\sigma \Z$.
\end{remark}

\section{Subrings of strongly $G$-graded rings}\label{Sec:Subrings}

In this section we give a description of certain subrings
of $G$-controlled rings.
We begin with the following result which generalizes \cite[Theorem 4(2)]{Azumaya1946}.

\begin{prop}\label{Prop:Subring}
If $R$ is a strongly $G$-graded ring which is $G$-controlled,
then there is a one-to-one correspondence between
submonoids of $G$ and unital subrings of $R$ containing $R_e$ given by
\begin{displaymath}
	\{\text{Submonoids of } G\} \ni H \, \stackrel{\phi}{\longmapsto} \, R_H = \oplus_{h\in H} R_h.
\end{displaymath}
In particular, if $R$ is a $G$-crossed product, then this occurs
if $R_e$ is simple and $C_R(R_e)=Z(R_e)$.
\end{prop}

\begin{proof}
If $H$ is a submonoid of $G$, then $R_H = \oplus_{h\in H} R_h$
is a unital subring of $R$, containing $R_e$. Hence $\phi$ is well-defined.
Moreover, it is clear that if $H_1 \neq H_2$ then $R_{H_1} \neq R_{H_2}$,
and this shows that $\phi$ is injective.

Let $S$ be a unital subring of $R$ containing $R_e$. Then $S$ is an $R_e$-bimodule and hence,
by the definition of a $G$-controlled ring,
there is a non-empty
subset $H\subseteq G$ such that $S=R_H$.
Take $g,h\in H$.
Using that $S$ is a ring and that $R$ is strongly $G$-graded,
we have 
$\{0\} \neq R_{gh} = R_g R_h \subseteq S$.
This shows that $gh\in H$ and hence $H$ is a subsemigroup of $G$.
From the fact that 
$R_e \subseteq S$ we get $e\in H$, and hence $H$ is a submonoid of $G$.
This shows that $\phi$ is surjective.
The last part follows from Corollary~\ref{MainResult:CrossedProducts}.
\end{proof}

\begin{corollary}\label{Cor:Subring}
Let $G$ be a finite group.
If $R$ is a strongly $G$-graded ring which is $G$-controlled,
then there is a one-to-one correspondence between
subgroups of $G$ and unital subrings of $R$ containing $R_e$ given by
\begin{displaymath}
	\{\text{Subgroups of } G\} \ni H \, \stackrel{\phi}{\longmapsto} \, R_H = \oplus_{h\in H} R_h.
\end{displaymath}
In particular, if $R$ is a $G$-crossed product, then this occurs
if $R_e$ is simple and $C_R(R_e)=Z(R_e)$.
\end{corollary}

\begin{remark}
Clearly, subrings of $R_e$ are also subrings of $R$,
but in general they can not be described by the above correspondence.
Take e.g. a skew group ring $A \rtimes_\sigma G$
and consider the subrings $Z(A)$ respectively $A^G = \{a\in A \mid \sigma_g(a)=a, \,\, \forall g\in G \}$.
Notice that $A^G=A$ if and only if $A \star_\sigma G$ is a group ring.
Hence, in Proposition~\ref{Prop:Subring} and Corollary~\ref{Cor:Subring}
the requirement ''subrings of $R$ containing $R_e$'' can not be relaxed.
\end{remark}

\section{Simple strongly $G$-graded rings and some open questions}\label{Sec:Disc}

By Proposition~\ref{GradedSimpleStronglyEquivalence}, $G$-controlled rings which
are strongly $G$-graded are necessarily simple.
In this section we shall discuss some known simplicity results for strongly $G$-graded rings
and see how they are related to our investigation of $G$-controlness. We will also present some open questions (see Section~\ref{subsec:openprob}).

The following result was shown by Van Oystaeyen (see \cite[Theorem 3.4]{FVO1984}).

\begin{theorem}[Van Oystaeyen, 1984]\label{FVOthm}
Let $R$ be a strongly $G$-graded ring
such that the morphism $G \to \Pic(R_e)$, defined by $g\mapsto [R_g]$, is injective.
If $R_e$ is a simple ring, then $R$ is a simple ring.
\end{theorem}

Using Lemma \ref{lemma:CentralizerPicard}
we get the following equivalent formulation of Theorem~\ref{FVOthm}.

\begin{prop}\label{Cor:Simplicity}
Let $R$ be a strongly $G$-graded ring
such that $C_R(R_e)=Z(R_e)$ holds.
If $R_e$ is a simple ring, then $R$ is a simple ring.
\end{prop}

It now becomes clear that Van Oystaeyen's result is in fact
a generalization of Azumaya's
result \cite[Theorem 4(1)]{Azumaya1946},
from skew group rings by finite groups to
general
strongly group graded rings. (In fact, it is even more general than Corollary~\ref{Cor:StrongSimplicity}.)
The following example shows that Proposition~\ref{Cor:Simplicity} does not necessarily hold
if we relax the assumption on the strong gradation.

\begin{exmp}
If $R$ is not strongly $G$-graded, then simplicity of $R_e$ and $C_R(R_e)=Z(R_e)$
are not enough to guarantee that $R$ be simple.
Indeed,
let $F$ be a field and let $\sigma : F \to F$ be a field automorphism
of infinite order.
We define a (not strongly) $\Z$-graded ring
$R=\oplus_{n \in \Z}R_n$,
with $R_n = F u_n$ for $n\geq 0$
and $R_n = \{0\}$ for $n<0$,
whose multiplication is defined by
$a u_n b u_m = a \sigma^n(b) u_{n+m}$
for $a,b\in F$ and $n,m\in \Z$.
Clearly, $F=R_e$ is simple and $C_R(R_e)=C_R(F)=F=Z(R_e)$.
Moreover, the ideal generated by $u_1$ is proper, hence $R$ is not simple.
Also notice that $R$ is not $\Z$-controlled.
\end{exmp}

\begin{remark}
Let $R$ be a $G$-graded ring.

(a) Suppose that the gradation on $R$ is left (or right) non-degenerate.
If $R_e$ is a simple ring, then $R$ is graded simple.

(b) $C_R(R_e)=Z(R_e)$ is not a necessary condition for simplicity of $R$.
To see this, consider e.g. the skew group ring $M_2(\R) \rtimes_\sigma \Z/2\Z$ in \cite[Example 4.1]{O2014}.
\end{remark}

Recall that a group $G$ is said to be \emph{hypercentral} if every non-trivial factor group of $G$ has a non-trivial center.
Hypercentral groups include e.g. all abelian groups.
The following result follows from \cite[Theorem 6]{Jespers1993} and
is a partial generalization of Van Oystaeyen's theorem (Theorem~\ref{FVOthm}).

\begin{prop}\label{thm:hypercentral}
Let $G$ be a hypercentral group and let $R$ be a strongly $G$-graded ring.
If $R$ is graded simple and $C_R(R_e)=Z(R_e)$ holds, then $R$ is a simple ring.
\end{prop}

\begin{proof}
Suppose that $R$ is graded simple and that $C_R(R_e)=Z(R_e)$ holds.
If we can show that $Z(R)$ is a field,
then by \cite[Theorem 6]{Jespers1993} we are done.

Take a non-zero $c\in Z(R) \subseteq C_R(R_e)=Z(R_e) \subseteq R_e$.
Clearly, $cR$ is a non-zero ideal of $R$.
Hence, by graded simplicity of $R$, we get $cR=R$.
From the gradation we conclude that $cR_e = R_e$.
In particular, $c$ is invertible in $R_e$.
One easily verifies that the inverse of $c$ belongs to $Z(R)$.
This shows that $Z(R)$ is a field.
\end{proof}

\begin{exmp}\label{Example:Matrix}
Consider the matrix ring $R=M_3(\C)$ equipped with the following gradation by $G=\Z/2\Z$.
\begin{displaymath}
	R_0=\left(
	\begin{array}{ccc}
	\C & 0 & \C \\
	0 & \C & 0 \\
	\C & 0 & \C
	\end{array}
	\right)
	\quad
		R_1=\left(
	\begin{array}{ccc}
	0 & \C & 0 \\
	\C & 0 & \C \\
	0 & \C & 0
	\end{array}
	\right)
\end{displaymath}
A short calculation shows that
\begin{displaymath}
	C_R(R_0)=
	\left\{
	\left(
	\begin{array}{ccc}
	a & 0 & 0 \\
	0 & b & 0 \\
	0 & 0 & a
	\end{array}
	\right) \Big\lvert \,\, a,b\in \C \right\}
	=Z(R_0).
\end{displaymath}
Another short calculation shows that $R_0$ has two non-trivial ideals;
\begin{displaymath}
	I=\left(
	\begin{array}{ccc}
	0 & 0 & 0 \\
	0 & \C & 0 \\
	0 & 0 & 0
	\end{array}
	\right)
	\quad \text{ and } \quad
		J=\left(
	\begin{array}{ccc}
	\C & 0 & \C \\
	0 & 0 & 0 \\
	\C & 0 & \C
	\end{array}
	\right).
\end{displaymath}
We notice that $R_1 I R_1 \subseteq J$ and $R_1 J R_1 \subseteq I$. Thus, $R$ is graded simple.
Using Proposition~\ref{thm:hypercentral} we retrieve a well-known fact: the matrix ring $R=M_3(\C)$ is simple.
Notice, however, that $R=M_3(\C)$ is not $\Z/2\Z$-controlled.
\end{exmp}

\subsection{Open questions}\label{subsec:openprob}

We shall now present some open questions which require further investigation.

\begin{remark}
Let $R$ be a strongly $G$-graded ring.
If $R_e$ is a division ring, then it follows almost immediately from the definition of a
strongly $G$-graded ring that $R$ is a $G$-crossed product.
We notice that the assumption on $R_e$ can be slightly relaxed.
In fact, if $R_e$ is a simple and artinian ring, then $R$ is a $G$-crossed product (see e.g. \cite[Lemma 1.1]{S1988}).
\end{remark}

It is easy to find examples of $G$-crossed products on which Theorem~\ref{FVOthm}
can be applied. Unfortunately, the literature does not seem to provide any example of a general strongly
$G$-graded ring (not a $G$-crossed product) satisfying the
conditions
of Theorem~\ref{FVOthm}.
Based on this, and in light of the above remark, we ask the following question.

\begin{question}\label{Q1}
Let $R$ be a strongly $G$-graded ring for which $R_e$ is a simple (and non-artinian) ring.
Is $R$ necessarily a $G$-crossed product?
\end{question}

If the answer to Question~\ref{Q1} is negative, then a natural follow-up question reads as follows.

\begin{question}\label{Q2}
Let $R$ be a strongly $G$-graded ring which is $G$-controlled.
Is $R$ necessarily a $G$-crossed product?
\end{question}

We want to know whether Van Oystaeyen's result (Theorem~\ref{FVOthm}) can be generalized
to situations when $R_e$ is not necessarily simple and ask the following.

\begin{question}\label{Q3}
Let $R$ be a strongly $G$-graded ring.
Suppose that $R$ is graded simple and that $C_R(R_e)=Z(R_e)$ holds.
Is $R$ necessarily simple?
\end{question}

\begin{remark}
Notice that Question~\ref{Q3} is known to have an affirmative answer in the following three cases:
\begin{enumerate}
	\item $R_e$ is simple (see Proposition~\ref{Cor:Simplicity});
	\item $R_e$ is commutative (see \cite[Theorem 6.6]{O2009});
	\item $G$ is a hypercentral group (see Proposition~\ref{thm:hypercentral}).
\end{enumerate}
\end{remark}


\begin{thebibliography}{AA}

\bibitem{Azumaya1946}
G. Azumaya,
New foundation of the theory of simple rings,
\emph{Proc. Japan Acad.} {\bf 22}(11) (1946), 325--332.

\bibitem{CameronSmith2015}
J. Cameron and R. R. Smith,
Bimodules in crossed products of von Neumann algebras,
\emph{Adv. Math.} {\bf 274} (2015), 539--561.


\bibitem{Dade1980}
E. C. Dade,
Group-Graded Rings and Modules,
\emph{Math. Z.} {\bf 174} (1980), 241--262.

\bibitem{Jespers1993}
E. Jespers,
Simple graded rings,
\emph{Comm. Algebra} {\bf 21}(7) (1993), 2437--2444.

\bibitem{Karpilovsky1987}
G. Karpilovsky,
\emph{The algebraic structure of crossed products},
North-Holland Mathematics Studies, 142.
Notas de Matemática [Mathematical Notes], 118. North-Holland Publishing Co., Amsterdam, 1987. x+348 pp. ISBN: 0-444-70239-3

\bibitem{Montgomery1980}
S. Montgomery,
\emph{Fixed rings of finite automorphism groups of associative rings},
Lecture Notes in Mathematics, 818. Springer, Berlin, 1980.
vii+126 pp. ISBN: 3-540-10232-9


\bibitem{NVO2004}
C. Nastasescu and F. Van Oystaeyen,
\emph{Methods of graded rings},
Lecture Notes in Mathematics, 1836.
Springer-Verlag, Berlin, 2004. xiv+304 pp.
ISBN: 3-540-20746-5

\bibitem{NVO2008}
E. Nauwelaerts and F. Van Oystaeyen,
Introducing crystalline graded algebras,
\emph{Algebr. Represent. Theory} {\bf 11}(2) (2008), 133--148.


\bibitem{NOP2016}
P. Nystedt, J. \"{O}inert and H. Pinedo,
Epsilon-strongly graded rings, separability and semisimplicity,
arXiv:1606.07592 [math.RA]

\bibitem{O2009}
J. \"{O}inert,
Simple group graded rings and maximal commutativity,
Operator Structures and Dynamical Systems (Leiden, NL, 2008), 159--175, \emph{Contemp. Math.} {\bf 503},
Amer. Math. Soc., Providence, RI, (2009).

\bibitem{O2014}
J. \"{O}inert,
Simplicity of skew group rings of abelian groups,
\emph{Comm. Algebra} {\bf 42}(2) (2014), 831--841.

\bibitem{OL2012}
J. \"{O}inert and P. Lundstr\"{o}m,
The ideal intersection property for groupoid graded rings,
\emph{Comm. Algebra} {\bf 40}(5) (2012), 1860--1871.


\bibitem{S1988}
P. Schmid,
Clifford theory of simple modules,
\emph{J. Algebra} {\bf 119}(1) (1988), 185--212.

\bibitem{FVO1984}
F. Van Oystaeyen,
On Clifford systems and generalized crossed products,
\emph{J. Algebra} {\bf 87}(2) (1984), 396--415.

\end{thebibliography}
\end{document}